\documentclass[12pt]{article}
\usepackage[margin=3cm]{geometry}
\usepackage{amsmath,amsthm,amssymb,amsfonts,mathtools,bm,bbm,array,float,mathdots,dsfont,mathrsfs,tensor,setspace,physics}
\usepackage{caption,subcaption,multirow,array,makecell,longtable,lscape,wrapfig,graphicx,comment,placeins}
\usepackage{hyperref,url}
\usepackage[utf8]{inputenc}
\usepackage{stmaryrd}
\usepackage{titling,lipsum}
\usepackage[table]{xcolor}
\usepackage[all]{xy}
\usepackage{fancyhdr}
\usepackage{pgfplots}
\usepackage{tikz-3dplot}
\usepackage[sort&compress,numbers]{natbib}
\usepackage{framed} 
\usepackage{tikz}
\usetikzlibrary{positioning, fit}
\usepackage{tabularx}
\usepackage[table]{xcolor}

\usepackage{xcolor}
\usepackage{tcolorbox}
\colorlet{shadecolor}{gray!15}
\renewenvironment{abstract}
  {\begin{center}\bfseries Abstract\end{center}%
   \begin{quote}\small}
  {\end{quote}\par\vspace{1em}}

\usepackage[T1]{fontenc}
\usepackage{xcolor,calc, blindtext}
\usepackage{tikz}
\usepackage{tikz-cd}
\usepackage{bbm}
\usepackage{float}

\theoremstyle{definition}
\newtheorem{definition}{Definition}[section] 
\newtheorem*{remark}{Remark}
\newtheorem{theorem}{Theorem}[section] 
\newtheorem{corollary}{Corollary}[theorem]
\newtheorem{lemma}[theorem]{Lemma}
\newtheorem{Remark}[theorem]{Remark}

\numberwithin{equation}{section}

\DeclareFontFamily{T1}{calligra}{}
\DeclareFontShape{T1}{calligra}{m}{n}{<->s*[1.44]callig15}{}
\DeclareMathAlphabet\mathcalligra   {T1}{calligra} {m} {n}
\DeclareMathAlphabet\mathzapf       {T1}{pzc} {mb} {it}
\DeclareMathAlphabet\mathchorus     {T1}{qzc} {m} {n}
\DeclareMathAlphabet\mathrsfso      {U}{rsfso}{m}{n}
 
\tdplotsetmaincoords{60}{115}
\pgfplotsset{compat=newest}

\begin{document}
\begin{center}
\rule{\textwidth}{0.4pt} \\[0.6cm]

{\Large \textbf{Reverse Isoperimetric Properties of \textit{Thick} $\lambda$-Concave Bodies in the Hyperbolic Plane}\par}

\vspace{0.6cm}

\rule{\textwidth}{0.4pt} \\[1.2cm]

{\textbf{Maria Esteban}\par}

\vspace{0.4cm}

\end{center}

\vspace{0.5cm}

\begin{abstract}
\sloppy
In this paper we address the reverse isoperimetric inequality for convex bodies with uniform curvature constraints in the hyperbolic plane $\mathbb{H}^2$. We prove that the\textit{ thick $\lambda$-sausage} body, that is, the convex domain bounded by two equal circular arcs of curvature $\lambda$ and two equal arcs of hypercircle of curvature  $1  / \lambda$, is the unique minimizer of area among all bodies  $K \subset \mathbb{H}^2$ with a given length and with curvature of $\partial K$ satisfying $1  / \lambda \leq \kappa \leq \lambda$ (in a weak sense). We call this class of bodies \textit{thick $\lambda$-concave} bodies, in analogy to the Euclidean case where a body is $\lambda$-concave if $0 \leq \kappa \leq \lambda$. The main difficulty in the hyperbolic setting is that the inner parallel bodies of a convex body are not necessarily convex. To overcome this difficulty, we introduce an extra assumption of thickness $\kappa \geq 1/\lambda$. 
\end{abstract}

\textit{Keywords:} reverse isoperimetric inequality, $\lambda$-concavity, $\lambda$-convexity, hyperbolic plane

\section{Introduction}

The reverse isopermietric problem is inspired from the classical isoperimetric problem which has been thoroughly studied in the past and generalized in a variety of contexts (see \cite{Ros2}). The classical problem concerns with finding the body with the largest volume for a given surface area. In $\mathbb{R}^n$  it is well known that the solution is an $n$-dimensional ball. The reverse isoperimetric problem addresses the opposite case. That is, it searches for the body that minimizes the volume for a given surface area.  This problem has a trivial solution as, for a given surface area, the $n$-dimensional volume can be arbitrary small (the \textit{flat pancake problem} \cite{ball}). 
\par
To avoid this degeneration, in this paper we restrict the reverse isoperimetric problem to convex bodies with uniform curvature constraints. 
Our focus will be on convex bodies lying in the hyperbolic plane $\mathbb{H}^2$, which we assume has constant curvature of $-1$. Recall that a set is called convex if the geodesic segment between any two points in it, lies inside. A convex body is a compact convex set with non-empty interior. In particular, we look at convex bodies $K \subset \mathbb{H}^2$ whose boundary curvature satisfies $1 / \lambda  \leq \kappa \leq \lambda$, in a weak sense (see below).  We call these \textit{thick $\lambda$-concave bodies}. This curvature constraint restricts the body from being too flat or too curved, ruling out trivial solutions such as the ``flat pancake shape''.  

\subsection{Some reverse isoperimetric results}

In this paper, we focus on $\lambda$-concave and $1/\lambda$-convex bodies. Throughout, we assume that $\lambda > 1$ and that the normal vector to the boundary $\partial K$ of a convex body $K$ points in the direction of convexity. In addition, given a convex body $K$,  $|\partial K|$ refers to the length (or surface area) of the boundary $\partial K$ and $|K|$ to the area (or volume) of $K$.

\begin{definition}[{$\lambda$-convex and $\lambda$-concave bodies in $\mathbb{H}^2$ \cite{kostya}}]
A convex body \( K \subset \mathbb{H}^2 \) is $\lambda$-convex if for each $  x \in \partial K $ there exists a neighborhood $ U_x \subset \mathbb{H}^2 $ and a closed curve  $\partial S_\lambda$ of constant curvature $\lambda$ passing through $x$ (see Section \ref{2.1}) such that
\[ U_x \cap \partial K \subset S_\lambda.\]
where  $ S_\lambda$ is the convex region bounded by  $\partial S_\lambda$. Similarly, $K$ is  $\lambda$-concave if for each $x \in \partial K $ there exists a neighborhood $ U_x \subset \mathbb{H}^2 $ such that 
\[ U_x \cap \partial S_{\lambda} \subset K.\]
Equivalently, if $\partial K$ is $C^2$-smooth, then $K$ is $\lambda$-convex when the geodesic curvature $\kappa_g$ satisfies 
$\kappa_g(x) \geq \lambda$ for all $ x \in \partial K$, and $\lambda$-concave when the geodesic curvature $\kappa_g$ satisfies $0 \leq \kappa_g(x) \leq \lambda $ for all $x \in \partial K$. 

\begin{definition} [Thick $\lambda$-concave bodies]
A convex body \( K \subset \mathbb{H}^2 \) is thick $\lambda$-concave if the geodesic curvature $\kappa_g$ satisfies $1  / \lambda \leq \kappa \leq \lambda$ (in a weak sense).
\end{definition}

\end{definition}

Some results on the isoperimetric problem of $\lambda$-convex and $\lambda$-concave bodies have already been obtained in $\mathbb{R}^n$.

\begin{theorem} [Reverse Isoperimetric Inequality for $\lambda$-concave Bodies] (\cite{kostya3}, see also \cite{saorin, nayar}). Let $n \geq 2$. Let $K \subset \mathbb{R}^n$ be a $\lambda$-concave body, and let $S \subset \mathbb{R}^n$ be a $\lambda$-sausage, i.e., the convex hull of two balls of radius $1/\lambda$. If $|\partial K| = |\partial S|$, then $|K| \geq |S|$, and equality holds if and only if $K$ is a $\lambda$-sausage.
\end{theorem}

Less is known for $\lambda$-convex bodies in $\mathbb{R}^n$. The following conjecture was formulated by Borisenko (see \cite{kostya3}): 
\newpage
\textbf{Conjecture} (The Reverse Isoperimetric Inequality for $\lambda$-convex bodies in $\mathbb{R}^n$). Let $n \geq 2$. Let $K \subset \mathbb{R}^n$ be a $\lambda$-convex body, and let $L \subset \mathbb{R}^n$ be a $\lambda$-convex lens, i.e., the intersection of two balls of radius $1/\lambda$. If $|\partial K| = |\partial L|$, then $|K| \geq |L|$, and equality holds if and only if $K$ is a $\lambda$-convex lens.

\begin{figure} [ht!] 
\centering
\includegraphics[width= 13.3cm, height=3.8cm]{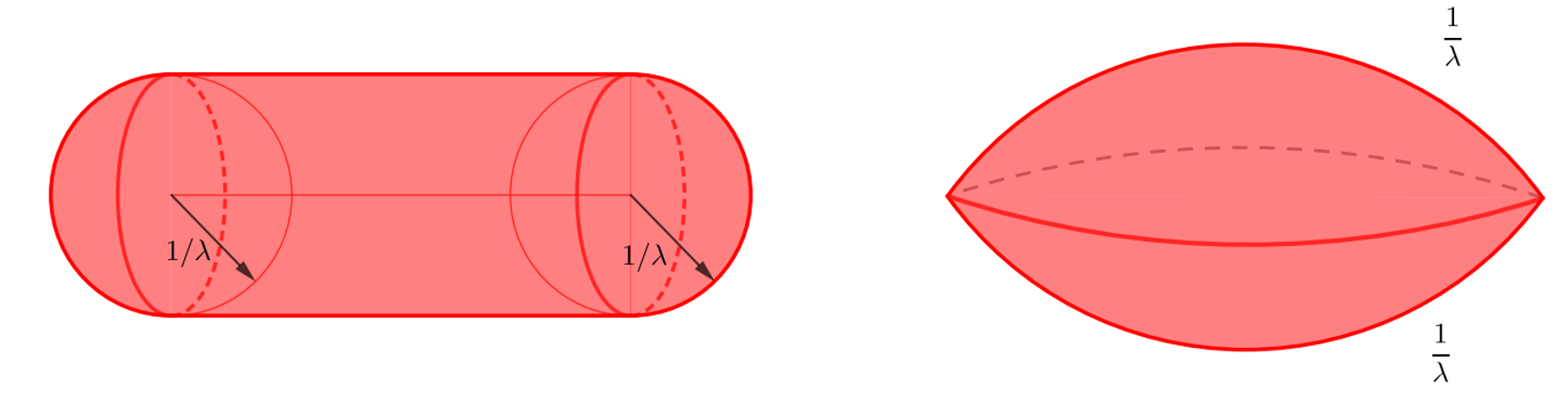}
\caption{ $\lambda$-sausage (left) and $\lambda$-convex lens (right) \cite{kostya3}}
\end{figure}

For $n=2$ the conjecture was proved in \cite{kostya2} (see \cite{fodor} for an alternative proof). For $n=3$ the conjecture was proved in \cite{kostya}. 
\newline

\subsection{Reverse isoperimetric inequality in $\mathbb{H}^2$}
We will first define the main object of interest in this paper, the \textit{thick $\lambda$-sausage}: 

\begin{definition} [Thick $\lambda$-sausage] \label{thickk}
A \textit{thick $\lambda$-sausage} is the  convex domain bounded by two equal circular arcs of curvature $\lambda$ and two equal arcs of hypercircle of curvature  $1  / \lambda$ (see Figure \ref{sausage}). 
\end{definition}

\begin{figure} [ht!] 
\centering
\includegraphics[scale=0.38,clip]{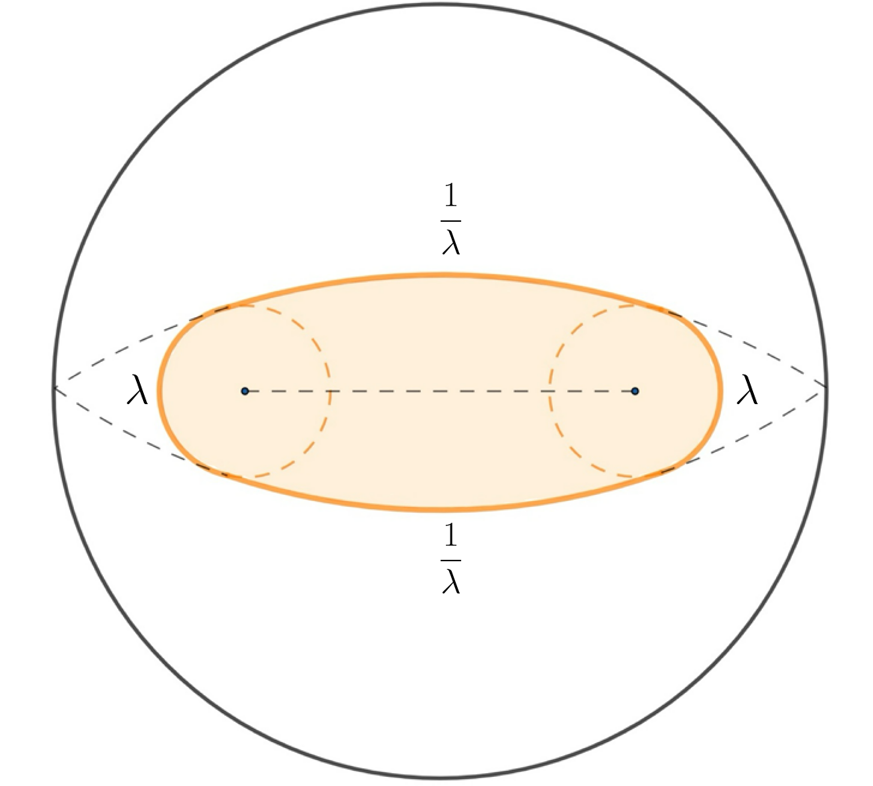}
\caption{Thick $\lambda$-sausage}
\label{sausage}
\end{figure}

In this paper we provide proofs of the theorems stated below. 
\begin{theorem}
[Reverse isoperimetric inequality for \textit{thick $\lambda$-concave bodies in $\mathbb{H}^2$}] \label{th1}
Let $K \subset \mathbb{H}^2$ be a \textit{thick $\lambda$-concave} body and let $S \subset \mathbb{H}^2$ be a \textit{thick $\lambda$-sausage}, as in Definition \ref{thickk}. If $|\partial K| = |\partial S|$, then  $|K| \geq |S|$. Moreover, equality holds if and only if $K$ is a \textit{thick $\lambda$-sausage}. 
\end{theorem}

\begin{theorem} [Lower bound for reverse isoperimetric inequality for \textit{thick $\lambda$-concave bodies in $\mathbb{H}^2$}] \label{th2}
Let $K \subset \mathbb{H}^2$ be a \textit{thick $\lambda$-concave} body. Then, 
\[ |K| \geq \frac{|\partial K|}{\lambda} + 2\pi \left( 1-\sqrt{
1-\frac{1}{\lambda^2}} \right).\]
Moreover, equality holds if and only if $K$ is a \textit{thick $\lambda$-sausage}.
\end{theorem}

\begin{Remark} In fact, if $K \subset \mathbb{H}^2(-c^2)$ is a \textit{thick $\lambda$-concave} body in the hyperpolic plane of curvature $-c^2$, $c>0$, then

\[ |K| \geq \frac{|\partial K|}{\lambda} + \frac{2\pi}{c^2} \left( 1-\sqrt{
1-\frac{c^2}{\lambda^2}} \right).\]

And, as $c \rightarrow 0$, the inequality approaches the classical Euclidean result in $\mathbb{R}^2$ proved in \cite{kostya3}. This inequality also complements the analogous results obtained in $\mathbb{S}^2$ \cite{kostya4}.
\end{Remark}

\subsection{Structure of the paper}
In Section \ref{2}, we provide some background theory on the curvature of convex bodies in the hyperbolic plane. In Section \ref{section 3}, we give the proof of Theorem \ref{th1} and Theorem~\ref{th2}.
The proof follows a similar approach as in the Euclidean case, where the structure of the inner parallel bodies is analyzed, but with the additional difficulty that in the hyperbolic plane the inner parallel bodies of a convex body are not necessarily convex.

\subsection{Acknowledgements}
I would like to express my gratitude to Kostiantyn Drach for supervising this project as part of the Barcelona Introduction to Mathematical Research 2024 (BIMR) summer program, organized by the Centre de Recerca \text{Matem\`atica} (CRM).

\section{Background Theory} \label{2}
\subsection{Curves of constant curvature in $\mathbb{H}^2$} \label{2.1}
In the hyperbolic space of curvature -$1$ there are multiple types of curves of constant geodesic curvature (see Figure \ref{summary}). This section provides the definition and key properties of such curves in the Poincar\'{e} Disk model and the Upper Half-Plane model, which are the relevant models for our analysis (see e.g  Section 35.2.4 in \cite{BZ} or Section 2.1 \cite{kostya}).

\begin{itemize}
    \item Circles: Circles of radius $R$ have geodesic curvature $ \lambda= \coth (R) >1$. Both in the Poincar\'{e} Disk model and in the Upper Half-Plane model, hyperbolic circles are Euclidean circles, even though, in general, the hyperbolic center does not coincide with the Euclidean center of the disk.
    \item Geodesics: By definition, the geodesic curvature of geodesics is 0. In the Poincar\'{e} Disk model, geodesics appear as arcs of circles that are orthogonal to the boundary of the disk or as straight lines passing through the center. In the Upper Half-Plane model, geodesics are represented as either semicircles orthogonal to the boundary (the real line) or vertical lines perpendicular to the boundary.
    \item Horocyles: A horocycle is a curve whose perpendicular geodesics converge asymptotically to the same direction. Horocyles have geodesic curvature $\lambda=1$. In the Poincar\'{e} Disk model, horocycles appear as Euclidean circles that are tangent to the boundary of the disk. The center of the horocyle is an ideal point (lies on the boundary of the hyperbolic space) and has ``infinite radius''. In the Upper Half-Plane model, horocycles appear as Euclidean circles tangent to the real line or as horizontal lines. 
 
    \item Hypercircles: Hypercircles are curves equidistant to a geodesic. They have a geodesic curvature $\lambda \in (0,1)$. In the Poincar\'{e} Disk model and Upper Half-Plane model, a hypercircle appears as an arc of a Euclidean circle that intersects the boundary of the model at two distinct points with an angle $0 < \alpha < \frac{\pi}{2}$. 
\end{itemize}

\begin{figure} [ht!] 
\centering
\includegraphics[width= 13.7cm, height=5.2cm]{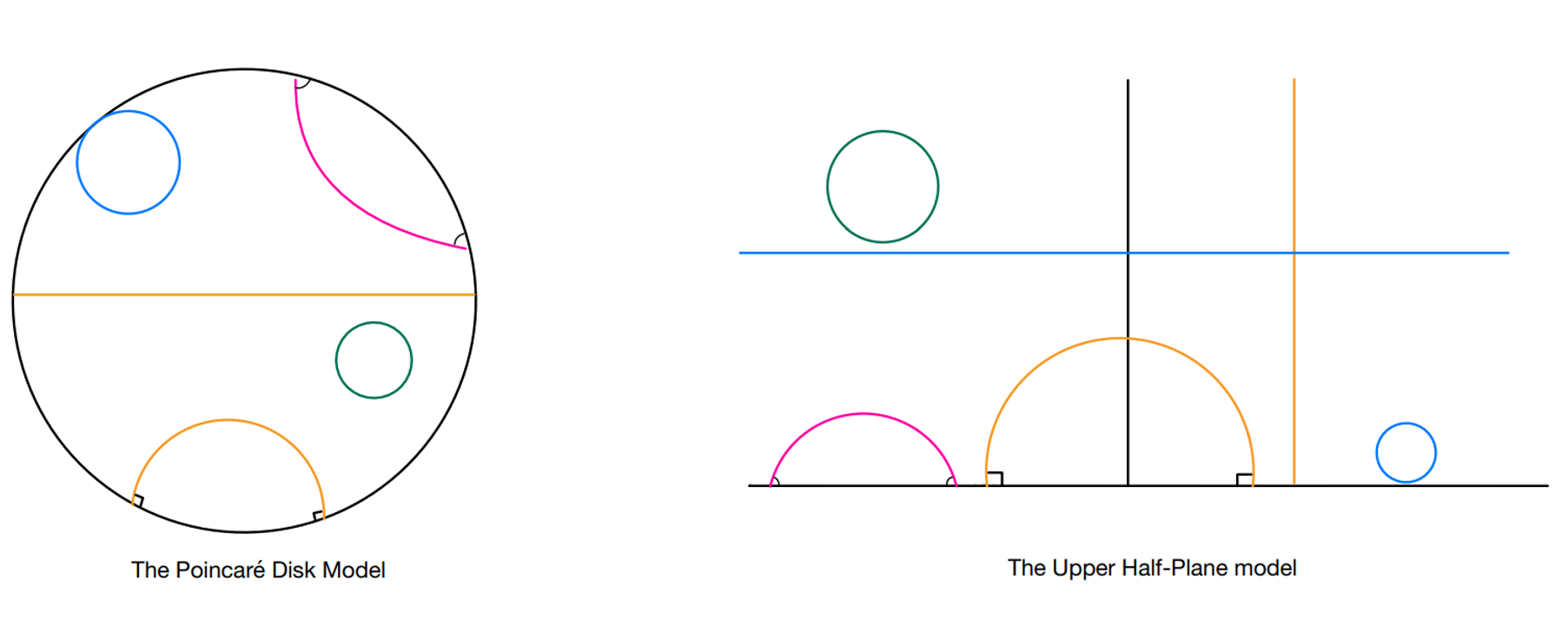}
\caption{ Circle (green), Geodesics (orange), Horocycles (blue) and Hypercircles (pink) }
\label{summary} 
\end{figure}

\newpage
\subsection{Some useful lemmas}
\begin{lemma} \label{hcurvature}
In the Poincar\'{e} Disk model, the geodesic curvature of a curve at the origin is $\frac{1}{2}$ of its Euclidean curvature at that point.
\end{lemma}

\begin{proof}
This result is standard, and we will only give a heuristic idea of the proof. To compute the geodesic curvature of a curve at $(0,0)$ in the Poincar\'{e} disk, we use parallel transport to quantify how much the curve deviates from a geodesic. We take the geodesic to be the horizontal diameter of the disk. Let $t$ denote the hyperbolic arclength parameter along the curve. Near the origin, the hyperbolic metric is approximately Euclidean up to a scaling factor of 2, so the Euclidean coordinates of the curve are approximately $\left( \frac{t}{2}, 0 \right)$ for $t \to 0$. Since parallel transport along the geodesic near the origin behaves approximately as in Euclidean space, the geodesic curvature at $(0,0)$ is half the Euclidean curvature at that point.
\end{proof}

\begin{lemma} \label{lastlemma}
Let $\beta$ be the angle at which the hypercircle $\Gamma$ intersects the boundary of the unit disk in the Poincar\'{e} Disk model, as illustrated in Figure~\ref{figura}. Then the curvature $\lambda$ of $\Gamma$ in $\mathbb{H}^2$ is given by $
\lambda = \cos(\beta).$
\end{lemma}

\begin{proof}
Let $\Gamma$ be a hypercircle containing the center of the disk (note that we can map any hypercircle to the center of this disk with a M\"obius transformation). We can then consider the construction shown in Figure  \ref{figura}, where the black circle denotes the boundary of the Poincar\'{e} Disk and the green circle the hypercircle $\Gamma$. 

Let $\delta$ denote the angle between the tangent to the unit circle and the tangent to the hypercircle at their point of intersection (red lines in Figure \ref{figura}). Since the radius of the circle is perpendicular to its tangent at the boundary point, we have that $\delta=\pi-\alpha$, which leads to $\delta+\beta=\pi$. This implies that $\beta=\alpha$. Using some trigonometric identities, we get that
\[ \sin \left( \frac{\pi}{2}- \alpha \right) = \frac{1}{2 R}.\]
Hence, 
\[2 \cos(\alpha)= \frac{1}{R}. \]

The hypercircle $\Gamma$ at (0,0) has Euclidean curvature $1/R=2\cos(\alpha)$. Therefore, from Lemma \ref{hcurvature}, the hyperbolic curvature of $\Gamma$ at (0,0) is $\cos(\alpha)=\cos(\beta)$. Since (0,0) is not a special point, every point of $\gamma$ has hyperbolic curvature $\cos(\beta)$. 
\end{proof}

\newpage
\begin{figure} [ht!] 
\centering
\includegraphics[width= 6.0 cm, height=6.8cm]{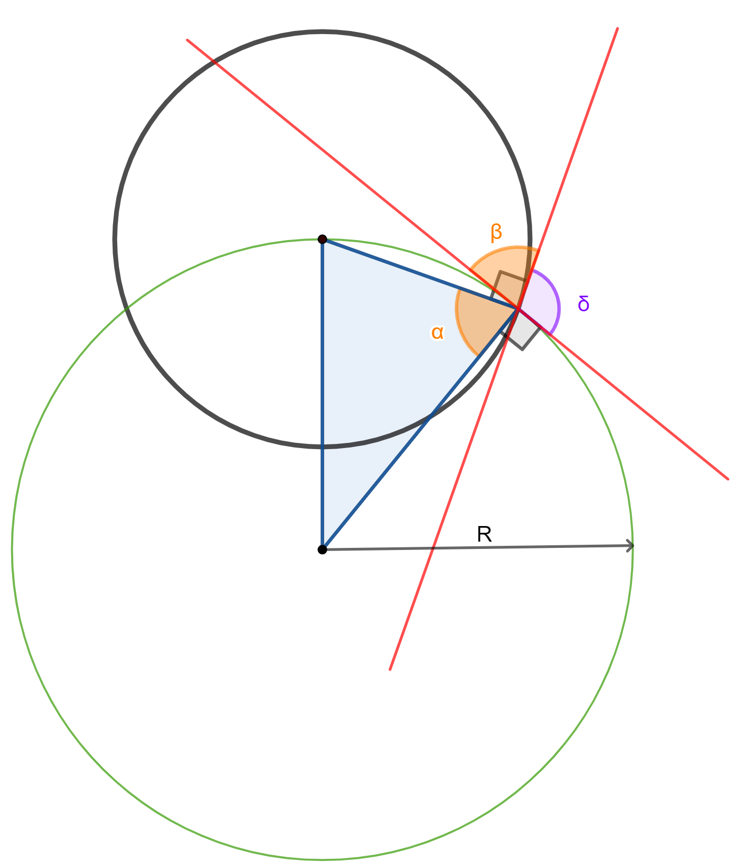}
\caption{ The proof of Lemma \ref{lastlemma} is constructed as follows. We consider the triangle formed by the center of the unit disk (black), the center of the hypercircle (green), and their point of intersection. By examining the properties of the tangent lines to the disk and hypercircle at the point of intersection, we can establish a relationship between the radius $R$, the curvature at (0,0) and $\beta$.}
\label{figura}
\end{figure}

\begin{lemma} \label{lemma2.5}
If all points of a hypercircle $\Gamma$ lie at a distance $R>0$ from a certain geodesic, then $\cos(\beta)=\tanh(R)$, where $\beta$ is the angle between $\Gamma$ and the boundary of the model. 
\end{lemma}

\begin{proof}
In the Upper-Half Plane model of hyperbolic space, hypercircles $\Gamma$ with respect to the geodesic line $x=0$ can be represented as straight lines intersecting the boundary with an angle $\beta < \frac{\pi}{2}$. In this model, the distance $R$ between $\Gamma$ and the geodesic is defined by the arc that intersects both lines and meets the boundary perpendicularly, as illustrated in Figure \ref{Figure5.1}.
\begin{figure} [ht!] 
\centering
\includegraphics[width= 9cm, height=5cm]{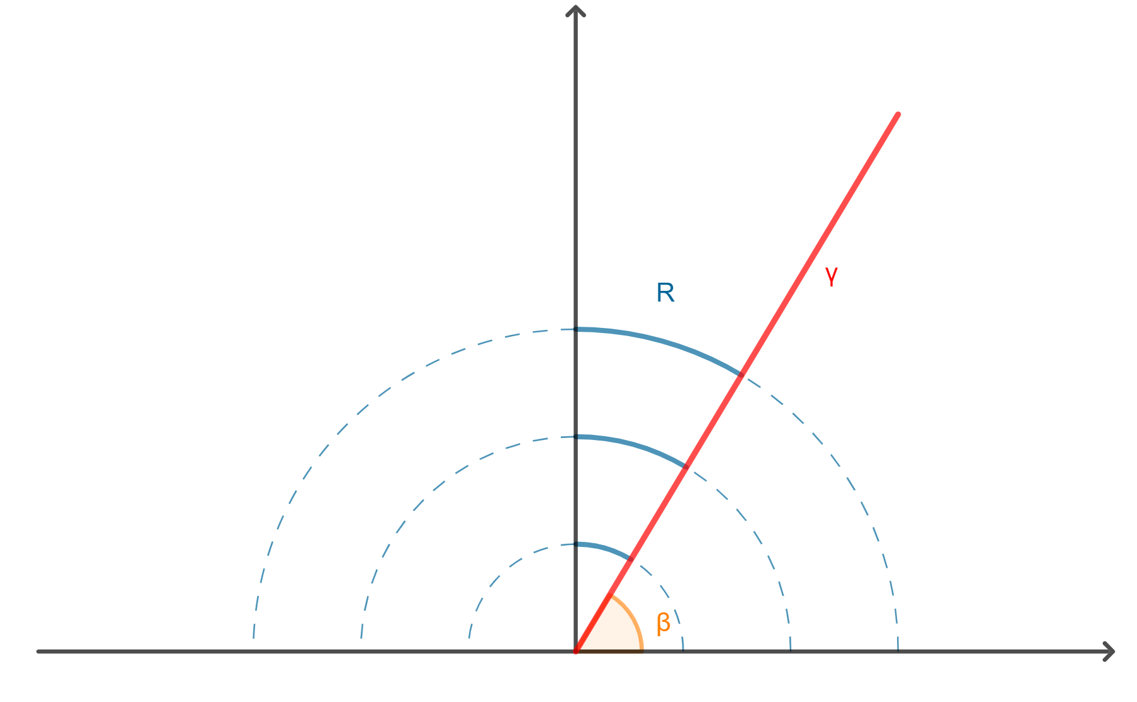}
\caption{ Hypercircle in the Upper Half-Plane }
\label{Figure5.1}
\end{figure}

\newpage
We can then parameterize this family of arcs as: 

\[ \sigma (t)= (A \cos(t), A\sin(t))\]

for $A>0$. Then, we can calculate $R$ by computing the hyperbolic length of $\sigma (t)$ from $\beta \leq t \leq \frac{\pi}{2}$,

\[ R= \int_{\beta}^{\frac{\pi}{2}}\frac{|\sigma^{'}(t)|}{\text{Im}(\sigma (t))} \text{dt} = \int_{\beta}^{\frac{\pi}{2}} \frac{1}{\sin(t)} \text{dt} = - \log\left( \tan\left( \frac{\beta}{2} \right) \right)
.\]

Then, 

\[\tanh(R)=\frac{1 - e^{-2R}}{1 + e^{-2R}} = \frac{1 - \tan^2\left( \frac{\beta}{2} \right)}{1 + \tan^2\left( \frac{\beta}{2} \right)} = \cos(\beta).
\]
Note that when $\Gamma  \subset \mathbb{H}^2(-c^2)$ then the curvature $\lambda$ of $\Gamma$ is normalized to \mbox{$\lambda = c \cdot \tanh(cR)$}

\end{proof}

\begin{lemma} \label{hyper}
Let $\Gamma$ be a hypercircle that lies at a distance $R$ from a geodesic. Suppose this geodesic connects the centers of two disks, each of radius $R$ and curvature $\lambda$ (see Figure \ref{sausage}). Then the geodesic curvature of $\Gamma$ is equal to $1/\lambda$.
\end{lemma}

\begin{proof}
The curvature of a circle of radius $R$ is
$\coth(R)$. From Lemma \ref{lemma2.5}, the curvature of the hypercircle is $ \tanh (R) = \tanh (\coth^{-1} (\lambda) )= 1/\lambda$, as required.  
\end{proof}

Note that Lemma \ref{hyper} defines the curvature of the hypercircles contained in the \textit{thick $\lambda$-sausage} defined in Definition \ref{thickk}.

\section{Proof of Theorems \ref{th1} and \ref{th2}} \label{section 3}

The proof of Theorem \ref{th1} is divided into 2 steps. The strategy is as follows: 
\begin{enumerate}
    \item Given a \textit{thick $\lambda$-concave} body $K$ we construct a convex body $L$ such that $ L \subset K$. We do that by using the notion of strict convexity from \cite{strictconvexity} (Subsection \ref{subsection 3.2}).
    \item We use the classical Steiner's formula to relate the area between $L$ and $K$, and their inner parallel bodies (Subsection \ref{3.3}). 
\end{enumerate}

\subsection{Construction of the inner parallel bodies} \label{subsection 3.2}
The next step of the proof constructs the convex inner parallel bodies of $K$. 

\begin{definition}[Inner parallel bodies $K_{-\rho}$ of $K$] Let $K$ be a convex body in $\mathbb{H}^2$. The \textit{inner parallel body} at distance $\rho \geq 0$ is defined as

\[
K_{-\rho}  := \{x \in \mathbb{H}^2 :B_{\rho}(x) \subset K\},
\]

where $B_{\rho}(x)$ is a ball of radius $\rho$ centered at $x$. The greatest number $\rho$ for which $K_{-\rho}$ is not empty is the \textit{inradius}, denoted as $in(K)$, of $K$. Equivalently, $in(K)$ is the radius of the largest ball contained in $K$. 
\end{definition}

\begin{lemma} \label{3.5} Let $L$ and $K$ be convex subsets of $\mathbb{H}^2$. If $L \subset K$, then $L_{-\rho} \subset K_{-\rho} $ for $0 \leq \rho \leq in(L)$.
\end{lemma}

\begin{proof}Let $x \in L_{-\rho}$. By the definition of inner parallel body $B_{\rho}(x) \subset L \subset K$, which implies that $x \in K_{-\rho}$. This is satisfied $\forall x \in L_{-\rho}$, so we conclude that $L_{-\rho} \subset K_{-\rho} $, as required. 
\end{proof}

\begin{Remark} Note that for all $N \in \mathbb{N}$, there exists a strictly convex body $\Omega \subset \mathbb{H}^2$ such that it has $N$ inscribed balls  not connected by continuous motion within $\Omega$. We can construct this body $H$ by taking of the convex hull of $N$ equal balls, suitably choosen, and considering the outer parallel body $H_{\rho}$ for a small $\rho > 0$. This ensures that $H_{\rho}$ is strictly convex. 
This phenomenon is not encountered in the Euclidean space, where there exist a unique inscribed ball for a strictly convex body $\Omega \subset \mathbb{R}^2$. 
\end{Remark}

\begin{Remark} Note that $K$ being convex does not guarantee that $K_{-\rho}$ is convex, as shown in Figure \ref{Figure9}. In that example, $K$ is the convex hull of two balls $B^1$ and $B^2$ of radius $R$. The inner parallel body $K_{-\rho}$ for $0< \rho < in(K)$ is the domain bounded by the hypercircles a distance $\rho$ from the geodesics and two balls of radius $r=R-\rho$. It is well known that the domain bounded by hypercircles is convex, therefore, there exist a geodesic $[x,y]$ for $x \in B^1_{-\rho}$ and $y \in B^2_{-\rho}$ such that $[x,y] \nsubseteq K_{-\rho}$ (represented by the orange line in Figure \ref{Figure9} below). This phenomenon is a consequence of the previous Remark. 
\end{Remark}

\begin{figure} [ht!] 
\centering
\includegraphics[scale=0.36,clip]{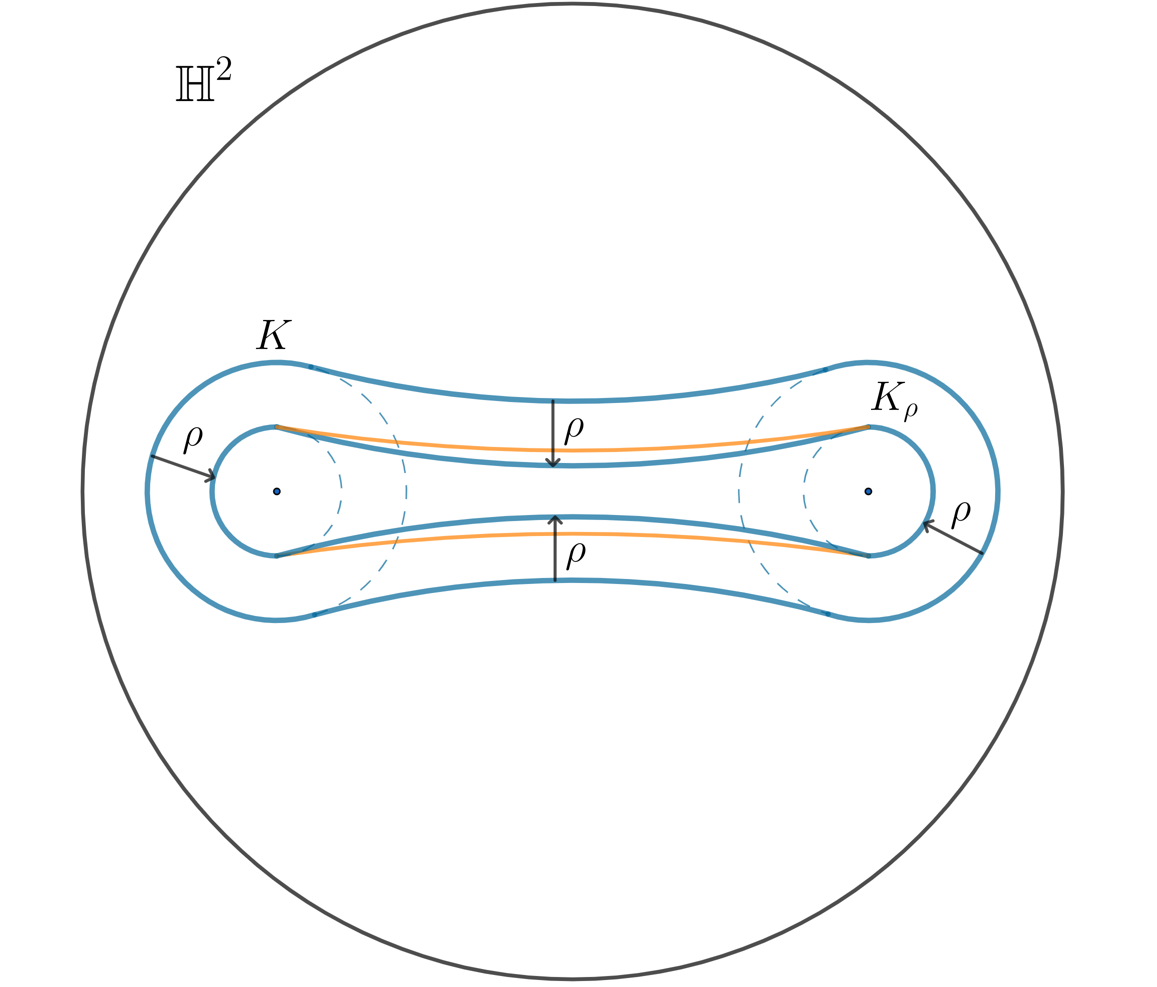}
\caption{Example of an inner parallel body of a convex body which is not convex}
\label{Figure9}
\end{figure}
\newpage
Because of this Remark, it is not immediate that the inner parallel body of $K$ is convex. We will proceed to prove this construction is indeed possible with the following lemma. 

\begin{lemma} \label{outerparallel} 
Let $K$ be a\textit{ thick $\lambda$-concave} body. Then, given 
$0 < \rho < in(K)$, there exists a non-empty convex set 
$L$ such that $K$ is the closed $\rho$-neighborhood of $L$. 
That is, $K$ is the union of all closed balls 
$B_\rho$ of radius $\rho$ whose centers lie in $L$. 
We can denote this as the Minkowski sum $K = L + B_\rho$.
\end{lemma}

\begin{proof}
This is a direct consequence of Theorem 1 in \cite{strictconvexity}, when taking $\epsilon=\tanh^{-1}(\lambda)$. Note that we can apply this result because all the bodies that we are working with have $C^{1,1}$ smooth boundary.     
\end{proof}

\begin{Remark} Lemma~\ref{outerparallel} provides a proof of Blaschke’s rolling theorem for \emph{thick $\lambda$-concave} bodies (see Figure \ref{FigureB}). In $\mathbb{R}^2$, the statement is as follows: let $\gamma$ be a closed, simple, smooth, convex curve in $\mathbb{R}^2$ such that for all $x \in \gamma$, $\kappa_g(x) \leq \lambda$. Then a ball of curvature $\lambda$ rolls along the boundary freely inside the region bounded by $\gamma$. This result was extended to $\mathbb{R}^n$ by \cite{B1,B2}, and later to Riemannian manifolds of bounded curvature in \cite{B3,B4,B5,B6}. Nevertheless, note that this statement does not hold in general for $\lambda$-concave domains in $\mathbb{H}^2$: consider the convex body $Q$ bounded by two equal circular arcs of curvature $\lambda$ and two equal arcs of a hypercircle of curvature $1/\lambda - \epsilon$, with $\epsilon > 0$. By construction, a ball of curvature $\lambda$ cannot roll freely inside $Q$ (see Figure~\ref{Figure9}).
\end{Remark}

\begin{figure} [ht!] 
\centering
\includegraphics[scale=0.47,clip]{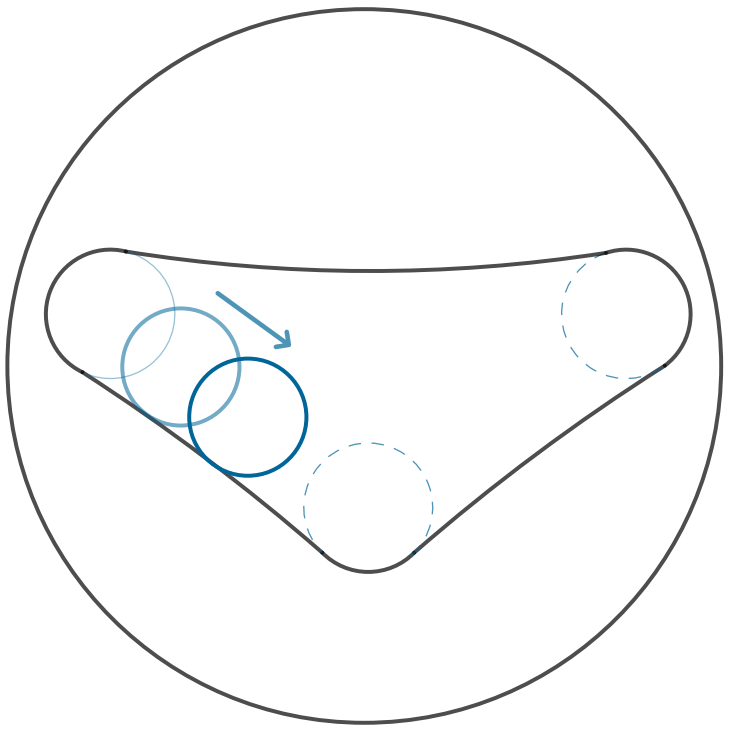}
\caption{Blaschke's rolling theorem for \textit{thick $\lambda$-concave} bodies in $\mathbb{H}^2$}
\label{FigureB}
\end{figure}

\newpage
\subsection{Steiner's formula in the Hyperbolic Space} \label{3.3}
We will now use Steiner's formula to relate the area of a body $K$ with the area of its inner parallel body. According to Steiner's classical result, let $\partial K$ be an arbitrary convex curve in the plane of length $|\partial K|$, bounding the area $|K|$.  
Let $\partial K_{\rho}$ be a curve parallel to $\partial K$ at a distance $\rho$, with length $|\partial K_{\rho}|$ and bounding the area $|K_{\rho}|$. Then,

\[ |\partial K_\rho| = |\partial K| + 2\pi \rho, \quad | K_\rho| = |K| + \rho |\partial K| + \pi \rho^2. \]

In \cite{steiner}, the author generalizes these formulae to simple, closed, and differentiable curves lying in spaces of constant curvature and satisfying certain curvature boundary constraints. Nevertheless, we are unable to apply this formula in our context, since $\partial K$ is not necessarily $C^{2}$. In \cite{sphericaleuclidean, KLAIN2007143}, the authors extend these formulae to compact convex sets in the sphere and in Euclidean space using valuations, and later \cite{hyperbolicsteiner} applies the same technique to generalize the result to hyperbolic space.

\begin{lemma} [Steiner formula for compact convex bodies in the hyperbolic space \cite{hyperbolicsteiner}] \label{steinerr} Let $K$ be a compact convex body in $\mathbb{H}^2$, and let $\rho\geq 0$, then    
\begin{equation} \label{steinerformula}
|K_{\rho}| = |\partial K| \sinh\left(\rho \right) + 2\pi \left(\cosh\left(\rho \right) - 1\right) + |K| \cosh\left(\rho \right).    
\end{equation}
\end{lemma}

\begin{proof}
The proof of this Lemma can be found in Corollary 4.5 of \cite{hyperbolicsteiner}. 
\end{proof}

\begin{corollary} \label{corollarymain}
Let $K$ be a \textit{thick $\lambda$-concave} body in $\mathbb{H}^2$, and let $0 \leq \rho < in(K)$, then    
\begin{equation} \label{steinerformulainner}
|K_{-\rho}| = -|\partial K| \sinh\left(\rho \right) + 2\pi \left(\cosh\left(\rho \right) - 1\right) + |K| \cosh\left(\rho \right). 
\end{equation}  
and 
\begin{equation}
 |\partial K_{-\rho}| = -\sinh(\rho) \big( |K| + 2\pi\big) + |\partial K| \cosh(\rho)  
\end{equation}
\end{corollary}

\begin{proof}
Let $0 \leq \rho < in(K)$ and $t>0$. By Lemma \ref{steinerr} we have that
\[
\begin{aligned}
|K_{\rho+t}| &= |\partial K| \sinh(\rho + t) 
  + 2\pi\big(\cosh(\rho + t) - 1\big) 
  + |K| \cosh(\rho + t) \\
&= |\partial K_\rho| \sinh(t) 
  + 2\pi\big(\cosh(t) - 1\big) 
  + |K_\rho| \cosh(t).
\end{aligned}
\]
Rearranging, we get
\[
\begin{aligned}
|\partial K_\rho|\sinh (t) + (|K_\rho|+2\pi)\cosh(t)
&= |\partial K|\big(\cosh(\rho)\,\sinh(t) + \sinh(\rho)\,\cosh(t)\big) \\
&\quad+ (|K|+2\pi)\big(\cosh(\rho)\,\cosh(t) + \sinh(\rho)\,\sinh(t)\big) \\
&= \cosh(t)\Big(\cosh(\rho)(|K|+2\pi) + \sinh(\rho)|\partial K|\Big) \\
&\quad+ \sinh(t)\Big(\sinh(\rho)(|K|+2\pi) + \cosh(\rho)|\partial K|\Big).
\end{aligned}
\]

Since this must hold for all $t$, the terms in front $\cosh(t)$ and $\sinh(t)$ of the formula above must be equal. That is: 

\[
\begin{pmatrix}
|K_\rho| + 2\pi \\[6pt]
|\partial K_\rho|
\end{pmatrix}
=
\begin{pmatrix}
\cosh(\rho) \quad & \quad \sinh(\rho) \\[6pt]
\sinh(\rho) \quad & \quad \cosh(\rho)
\end{pmatrix}
\begin{pmatrix}
|K| + 2\pi \\[6pt]
|\partial K|
\end{pmatrix}.
\]

Rearranging, we get to 

\begin{equation} \label{matrixx1}
\begin{pmatrix}
|K| + 2\pi \\[6pt]
|\partial K|
\end{pmatrix}
=
\begin{pmatrix}
\cosh(\rho) & -\sinh(\rho) \\[6pt]
-\sinh(\rho) & \cosh(\rho)
\end{pmatrix}
\begin{pmatrix}
|K_\rho| + 2\pi \\[6pt]
|\partial K_\rho|
\end{pmatrix}.
\end{equation}

Now, from Lemma \ref{outerparallel} we know that a \textit{thick $\lambda$-concave} body $K$ is the outer parallel body of some convex body $K
_{-\rho}$, for $0 \leq \rho < in(K)$ . Hence, expression (\ref{matrixx1}) becomes 

\begin{equation} \label{matrixx}
\begin{pmatrix}
|K
_{-\rho}| + 2\pi \\[6pt]
|\partial K
_{-\rho}|
\end{pmatrix}
=
\begin{pmatrix}
\cosh(\rho) & -\sinh(\rho) \\[6pt]
-\sinh(\rho) & \cosh(\rho)
\end{pmatrix}
\begin{pmatrix}
|K
| + 2\pi \\[6pt]
|\partial K
|
\end{pmatrix}
\end{equation}
as required. 
\end{proof}

Note that the result of Corollary \ref{corollarymain} is analogous to the result obtained in \cite{steiner} for bodies in $\mathbb{H}^2$ with $C^2$ smooth boundary. 
\par
We are now ready to finalize the proof of Theorems \ref{th1} and \ref{th2}. 
\vspace{2mm}

\textbf{Proof of Theorem \ref{th1}:} Let $K \subset \mathbb{H}^2$ be a \textit{thick $\lambda$-concave} body and $S \subset \mathbb{H}^2$ be a \textit{thick $\lambda$-sausage}. Rearranging formula (\ref{steinerformulainner}) we get 

\begin{equation} \label{steiner2}
| K|= |K_{-\rho}|\operatorname{sech}(\rho)  + |\partial K| \tanh\left(\rho \right) + 2\pi \left(\operatorname{sech}(\rho) -1 \right).
\end{equation}

Let  $\rho_\lambda = \coth^{-1}(\lambda)$. Then, $S_{-\rho_\lambda}$ is a geodesic segment with  $|S_{-\rho_\lambda}|=0$. Then, \(\mbox{$0=|S_{-\rho_\lambda}|\leq |K_{-\rho_\lambda}|$}\). By the assumption of Theorem \ref{th1}, we have \( |\partial S| = |\partial K| \). Then,by  applying formula (\ref{steiner2}), it follows that \( |S| \leq |K| \), as required.
\newline
To prove the other direction, notice that if $|K_{-\rho_\lambda}|=0$ then  $|K_{-\rho_\lambda}|$ is a convex region with null area and, hence, a geodesic segment. Moreover, \(K_{-\rho_\lambda}\) is a geodesic segment whenever \(K\) is a \emph{thick \(\lambda\)-sausage body}, as defined in Definition~\ref{thickk}.

\vspace{2mm}

\textbf{Proof of Theorem \ref{th2}:} Substituting $\rho = \coth^{-1} (\lambda)$ in equation (\ref{steiner2}), we get that the area of the \textit{thick $\lambda$-sausage} is

\begin{equation} \label{equality}
|S| = \frac{|\partial S|}{\lambda} + 2\pi \left( 1-\sqrt{
1-\frac{1}{\lambda^2}} \right). 
\end{equation}

Therefore, for the area of a  \textit{thick $\lambda$-concave} body $K$ satisfies
\begin{equation} \label{lastequality}
|K| \geq \frac{|\partial K|}{\lambda} + 2\pi \left( 1-\sqrt{
1-\frac{1}{\lambda^2}} \right).
\end{equation}

Applying the appropriate scaling, when $K \subset \mathbb{H}^2(-c^2)$,  (\ref{lastequality}) becomes

\[ |K| \geq \frac{|\partial K|}{\lambda} + \frac{2\pi}{c^2} \left( 1-\sqrt{
1-\frac{c^2}{\lambda^2}} \right).\]

It is easy to check (applying the L'H\^{o}pital Rule twice) that as $c \rightarrow 0$ from the left, equation (\ref{equality}) approaches the result in the Euclidean space proved in \cite{kostya3}. This concludes the proof of one direction of Theorem \ref{th2}. The other direction is trivial.

\bibliographystyle{unsrt}
\bibliography{main.bib}

\end{document}